\documentclass[oneside,english]{amsart}
\usepackage[T1]{fontenc}
\usepackage[latin9]{inputenc}
\usepackage{color}
\usepackage{babel}
\usepackage{array}
\usepackage{amsthm}
\usepackage{amssymb}
\usepackage{graphicx}
\usepackage{esint}
\usepackage[unicode=true,pdfusetitle,
 bookmarks=true,bookmarksnumbered=false,bookmarksopen=false,
 breaklinks=false,pdfborder={0 0 0},backref=false,colorlinks=true]
 {hyperref}
\usepackage{breakurl}

\makeatletter

\providecommand{\tabularnewline}{\\}

\numberwithin{equation}{section}
\numberwithin{figure}{section}
  \theoremstyle{definition}
  \newtheorem*{problem*}{\protect\problemname}
\theoremstyle{plain}
\newtheorem{thm}{\protect\theoremname}[section]
  \theoremstyle{plain}
  \newtheorem{conjecture}[thm]{\protect\conjecturename}
  \theoremstyle{definition}
  \newtheorem{defn}[thm]{\protect\definitionname}
  \theoremstyle{definition}
  \newtheorem{example}[thm]{\protect\examplename}
  \theoremstyle{remark}
  \newtheorem{rem}[thm]{\protect\remarkname}
  \theoremstyle{plain}
  \newtheorem{prop}[thm]{\protect\propositionname}
  \theoremstyle{remark}
  \newtheorem*{rem*}{\protect\remarkname}
  \theoremstyle{plain}
  \newtheorem{lem}[thm]{\protect\lemmaname}
  \theoremstyle{plain}
  \newtheorem{cor}[thm]{\protect\corollaryname}

\usepackage{ae}
\usepackage{aecompl} 

\makeatother

  \providecommand{\conjecturename}{Conjecture}
  \providecommand{\corollaryname}{Corollary}
  \providecommand{\definitionname}{Definition}
  \providecommand{\examplename}{Example}
  \providecommand{\lemmaname}{Lemma}
  \providecommand{\problemname}{Problem}
  \providecommand{\propositionname}{Proposition}
  \providecommand{\remarkname}{Remark}
\providecommand{\theoremname}{Theorem}

\begin{document}

\title[The LCDC and vertical surface area in warped products]{The Log-Convex Density Conjecture and vertical surface area in warped
products. }

\author{Sean Howe}
\begin{abstract}
We examine the vertical component of surface area in the warped product
of a Euclidean interval and a fiber manifold with product density.
We determine general conditions under which vertical fibers minimize
vertical surface area among regions bounding the same volume and use
these results to conclude that in many such spaces vertical fibers
are isoperimetric. Our main hypothesis is that the surface area of
a fiber be a convex function of the volume it bounds. We apply our
results in the specific case of $\mathbb{R}^{n}-\{0\}$ realized as
the warped product $(0,\infty)\times_{r}S^{n-1}$, providing many
new examples of densities where spheres about the origin are isoperimetric,
including simple densities with finite volume, simple densities that
at the origin are neither log-convex nor smooth, and non-simple densities.
We also generalize the results of Kolesnikov and Zhdanov\emph{ }on
large balls in $\mathbb{R}^{n}$ with increasing strictly log-convex
simple density. We situate our work in relation to the Log-Convex
Density Conjecture of Rosales \emph{et al. }and the recent work by
Morgan, Ritoré, and others on formulating a generalized log-convex
density/stable spheres conjecture. 
\end{abstract}
\maketitle
\emph{MSC Classification: 53A10, 49Q20. Key words: Manifolds with
density, Log-Convex Density Conjecture, isoperimetric, warped product,
product density, stable spheres.}

Sean Howe (seanpkh@gmail.com), University of Chicago Math Department,
5734 S. University Avenue, Chicago IL 60615. 

\tableofcontents{}

\section{Introduction}

In this paper we consider the isoperimetric problem in manifolds with
density: 
\begin{problem*}
In a Riemannian manifold $M$ equipped with a positive function $\Psi_{S}$
weighting surface area and a positive function $\Psi_{V}$ weighting
volume, which region has the least weighted surface area among all
regions of weighted volume $V_{0}$? 
\end{problem*}
In this paper a hypersurface is always rectifiable and a region always
has rectifiable boundary. A region is called isoperimetric if it has
minimal weighted surface area among all regions of the same weighted
volume. A hypersurface is called isoperimetric if it bounds an isoperimetric
region. We note that there is no a priori guarantee that an isoperimetric
region exists, and indeed there are simple examples of spaces with
density where isoperimetric regions do not exist (see, e.g., \cite[Prop. 7.3]{diazetal}).
In the rest of this paper we will omit the term ``weighted'' before
surface area and volume and refer to regular surface area and volume
as ``unweighted.\textquotedbl{} For a general reference on manifolds
with density see \cite{morganmanwithdens} or better, \cite[Ch. 18]{morgangmt}. 

The function $\Psi_{S}$ from above will be referred to as the \emph{surface
density} or \emph{perimeter density} and the function $\Psi_{V}$
as the \emph{volume density} (or, sometimes on 2-dimensional manifolds,
\emph{area density}). By a conformal change of metric one can always
take $\Psi_{S}=\Psi_{V}$ (see Proposition \ref{pro:Conformal-To-Simple-Basic}),
which is referred to as \emph{simple density}, however it is often
more convenient to vary the density than to vary the metric and so
we allow $\Psi_{S}$ and $\Psi_{V}$ to differ. Other interesting
and sometimes useful special cases are $\Psi_{V}=1$ (\emph{surface
density} or \emph{perimeter density}, see e.g. \cite{bettaEtAl},
\cite[Thm. 7.4]{diazetal}) and $\Psi_{S}=1$ (\emph{volume density},
see e.g. \cite[proof of Thm. 4.8]{diazetal}, \cite{morganBlogLogConvex}). 

Most work in manifolds with density has focused on $\mathbb{R}^{n}$
with simple density (see e.g. \cite[Ch. 18]{morgangmt}, \cite{carrolletal},
\cite{rosalesdensity}, \cite{englesteinetal}, \cite{maurmannMorgan},
\cite{caneteetal}, \cite{dahlbergetal}, \cite{kolesnikovRadiallySymmetric},
\cite{RatzikinWedgeDomains}, \cite{diazetal}). Of particular interest
is radial simple density (where the density is a function of the radius),
where interest has centered around the following conjecture of Rosales
\emph{et al.} \cite[Conj. 3.12]{rosalesdensity}:
\begin{conjecture}[{Log-Convex Density Conjecture \cite[Conj. 3.12]{rosalesdensity} }]
\label{con:LCDC}In $\mathbb{R}^{n}$, $n\geq2$ with radial log-convex
simple density, balls about the origin are isoperimetric for every
volume.
\end{conjecture}
For a simple density on $\mathbb{R}^{n}$, log-convexity on $\mathbb{R}^{n}-\{0\}$
is equivalent to the stability of all spheres about the origin, a
necessary condition for balls about the origin to be isoperimetric.
As pointed out by Morgan \cite{morganBlogLogConvex,Morgan-PubLogConvexDensity},
however, a further regularity condition at the origin is necessary
to avoid examples such as $\mathbb{R}^{n}$ with simple density $e^{r^{2}-2r+2}$
where for small volumes isoperimetric regions are approximate balls
centered on the unit sphere. The Log-Convex Density Conjecture, then,
posits that log-convexity at the origin would be sufficient to avoid
this and any other counterexample. Chambers \cite{chambersLCDCproof}
has recently announced a proof of Conjecture \ref{con:LCDC} along
with an optimal uniqueness statement, thus completely characterizing
the isoperimetric regions in $\mathbb{R}^{n}$ with radial log-convex
simple density. 

However, there are also natural examples of densities log-convex on
$\mathbb{R}^{n}-\{0\}$ but not at the origin where spheres about
the origin are nonetheless isoperimetric, though bounding volume at
infinity instead of at the origin. For example, Diaz \emph{et al.}
\cite[ Prop 7.5]{diazetal} show that in $\mathbb{R}^{n}-\{0\}$ with
simple density $r^{p},\; p<-n$, spheres about the origin are isoperimetric
bounding volume at infinity. One might ask whether there is a more
general conjecture including both the classic Log-Convex Density Conjecture
and these examples. One can also ask whether it is possible to formulate
a sufficient condition when spheres are stable but the perimeter and
volume density are allowed to differ, as suggested by Morgan \cite{morganBlogLogConvex,MorganBlogVariationalFormulae,Morgan-PubLogConvexDensity},
who also calculates the stability condition explicitly in this case.

Any such formulation would need to overcome multiple obstacles --
in addition to the above example, if we remove log-convexity at the
origin then we also get densities such as $\mathbb{R}^{n}$ with simple
density $r^{-p},\;0<p\leq n$ where spheres are stable but  isoperimetric
regions do not even exist (see \cite[Prop. 7.3]{diazetal}). While
it is still unclear what the most general statement of the conjecture
should be, our Theorem \ref{thm: Isoperimetric Regions - General},
which states that spheres about the origin are isoperimetric in $\mathbb{R}^{n}-\{0\}$
with any radial density such that the surface area of such spheres
is a convex function of the volume they bound and satisfying some
additional minor hypotheses, allows us to give several interesting
new examples of densities on $\mathbb{R}^{n}-\{0\}$ for which spheres
about the origin are isoperimetric. These examples include simple
densities with finite volume, simple densities that are neither log-convex
nor smooth at the origin, and non-simple densities (see Example \ref{exa:IsoperimExamples}).
In particular, the existence of large families of densities on $\mathbb{R}^{n}$
which are neither log-convex nor smooth at the origin for which stable
spheres are nonetheless isoperimetric complicates the formulation
suggested in \cite{morganBlogLogConvex,Morgan-PubLogConvexDensity}.
In Theorem \ref{thm:SingleFiberGeneralMinimizeTanIso}, Corollary
\ref{cor:LargeVerticalFibers}, and Corollary \ref{cor:LargeBalls-SimpleDensity}
we generalize a result of Kolesnikov and Zhdanov \cite[Prop. 4.7]{kolesnikovRadiallySymmetric}
on large balls about the origin in $\mathbb{R}^{n}$ with strictly
log-convex increasing simple density. We use this generalization to
show that for a large family of densities on $\mathbb{R}^{n}$ for
which spheres about the origin are stable they are also isoperimetric
for large volumes, and we provide several examples of densities where
these results apply, including the above-mentioned example of $\mathbb{R}^{n}$
with simple density $e^{r^{2}-2r+2}$, thus giving an example of a
space where all spheres about the origin are stable but only some
are isoperimetric (Example \ref{exa:LargeBalls}). 

We obtain our results by analysing the component of surface area tangential
to spheres about the origin. This is a special case of the concept
of vertical surface area, which we define for a rectifiable hypersurface
in the warped product of a real interval and a Riemannian manifold:
\begin{defn}
\label{def:VerticalSurfaceArea}Let $L$ be a Riemannian manifold
of dimension $n-1$ with metric $dl^{2}$ and let $Z$ be the interval
$(A,B),\; A<B\in\mathbb{R}\cup\pm\{\infty\}$ with the usual metric
$dr^{2}$. Consider the warped product $Z\times_{g}L$ with continuous
warp factor $g$ giving metric $dr^{2}+g(r)^{2}dl^{2}$ and continuous
surface density $\Phi_{S}$ and volume density $\Phi_{V}$. The \emph{vertical
surface area }of a rectifiable hypersurface $H$ in the warped product
with density $Z\times_{g}L$ is 
\[
|H|_{Vert}=\int_{H}|\vec{n}\cdot\vec{r}|\Phi_{S}d\mathcal{H}^{n-1}
\]
where at any point $\vec{r}$ is the positively oriented unit vector
perpendicular to $L$, $\vec{n}$ is the outward normal vector to
the surface, and $d\mathcal{H}^{n-1}$ is the $n-1$ dimensional Hausdorff
measure on $H$ inherited from the warped product $Z\times_{g}L$
without density. Locally where $H$ is a graph over $L$, 
\[
|H|_{Vert}=\int\Phi_{S}(r(l),l)g(r(l))^{n-1}dL
\]

\end{defn}
Note that where $H$ is parallel to horizontal fibers the contribution
to vertical surface area is $0$ and so we can always calculate vertical
surface area using only the local formula. Further, the definition
of vertical surface area gives trivially the expected inequality
\[
|H|_{Vert}\leq|H|,
\]
where $|H|$ is the surface area of $H$, with equality only when
$H$ is a union of vertical fibers. 

We note that general warped products with density have already appeared
in \cite{morgansym}. 
\begin{example}
We can realize $\mathbb{R}^{n}-\{0\}$ with Euclidean metric as the
warped product $(0,\infty)\times_{r}S^{n-1}$. In this context, we
often refer to vertical surface area as \emph{tangential surface area
}because it is the component of surface area tangential to spheres
about the origin. Analysis of tangential surface area was used in
\cite[Prop. 4.3]{carrolletal} and \cite[Thm. 7.4]{diazetal} to prove
that in $\mathbb{R}^{n}-\{0\}$ with density $r^{p},\; p<-n$ , spheres
about the origin minimize tangential surface area and are thus isoperimetric
(bounding volume at infinity), and similarly in \cite[Prop. 7.5]{diazetal}
to give a new proof of the result of Betta, et al. \cite[Thm. 4.3]{bettaEtAl}
that in $\mathbb{R}^{n}$ with certain surface densities spheres about
the origin are isoperimetric. Section \ref{sec:Vertical-surface-area-Warped-Products}
generalizes and refines these ideas.\end{example}
\begin{rem}
One can also study the weaker inequality $\int_{H}(\vec{n}\cdot\vec{r})\cdot f\cdot\Phi_{S}d\mathcal{H}^{n-1}<|H|$
where $f$ is a function on the real interval with $|f|\leq1$. This
is the approach taken by Kolesnikov and Zhdanov \cite{kolesnikovRadiallySymmetric}
in their Proposition 6.7 and the surrounding discussion in the setting
of $\mathbb{R}^{n}$ with increasing simple radial density. Using
this formula and the divergence theorem they show that for density
$e^{\phi(r)}$ with $\phi$ convex, radially symmetric and superlinear
(e.g. $e^{r^{\alpha}},\;\alpha>1$), large balls about the origin
are isoperimetric. In Corollaries \ref{cor:LargeVerticalFibers} and
\ref{cor:LargeBalls-SimpleDensity} we generalize this result using
Theorem \ref{thm:SingleFiberGeneralMinimizeTanIso} which gives conditions
on when a single vertical fiber in a warped product with density minimizes
vertical surface area, and in Example \ref{exa:LargeBalls} we give
several specific densities where our result applies. Our proof turns
on the use of comparison spaces with different surface densities as
developed in Section \ref{sec:Vertical-surface-area-Warped-Products}.
The weighting factor used by Kolesnikov and Zhdanov is similar, however,
it is not exactly analogous -- their weighting factor can be negative
valued whereas our comparison spaces always have positive surface
densities. The difference results primarily from the absence of an
absolute value around the term $\vec{n}\cdot\vec{r}$ in their approach
which thus gives a weaker inequality but allows the application of
the divergence theorem. 
\end{rem}
The most closely related results are those of Kolesnikov and Zhdanov
\cite[Sec. 6]{kolesnikovRadiallySymmetric} on large balls in $\mathbb{R}^{n}$
with increasing strictly log-convex density and those of Diaz et al.
\cite[Sec. 7]{diazetal} on $\mathbb{R}^{n}$ with density $r^{p},$
$p<0$ , both of which are generalized by this work. We note also
that both Montiel \cite{Montiel-StableCMC,MontielUnicityCMC} and
Rafalski \cite{rafalski-relativeisoperimwarpedproducts} have obtained
related results for graphs over horizontal regions in warped products. 

In Section \ref{sec:Vertical-surface-area-Warped-Products} we prove
the most general versions of our theorems in the context of warped
products with density. In Section \ref{sec:RnWithRadialDensity} we
apply these result to the most interesting case of $\mathbb{R}^{n}$
with radial density and give many specific examples. Although the
results of Section \ref{sec:RnWithRadialDensity} are stated only
for radial densities, they apply equally with product surface and
volume densities of the form $\Psi_{S}(r)\Phi(\Theta)$, $\Psi_{V}(r)\Phi(\Theta)$,
and even more generally to product densities of this form on any warped
product $(0,\infty)\times_{r}K$ with $K$ a compact Riemannian manifold
(see the beginning of Section \ref{sec:RnWithRadialDensity}).

\subsection*{Acknowledgements}

We would like to thank Frank Morgan for his invaluable insight and
advice throughout the preparation of this paper, as well as the 2010
SMALL Geometry Group \cite[Sec. 6]{LiEtAlG11TilingsandIsoperimetry}
for inspiring in us a renewed interest in the Log-Convex Density Conjecture
with their work on the density $e^{r}$. We would also like to thank
an anonymous referee for their careful reading and helpful comments.
The author was supported by an Erasmus Mundus scholarship and enrolled
in the ALGANT integrated master course during part of the preparation
of this paper.

\section{\label{sec:Vertical-surface-area-Warped-Products}Minimization of
vertical surface area in warped products}

Proposition \ref{pro:Conformal-To-Simple-Basic} is a well-known result.
\begin{prop}
\label{pro:Conformal-To-Simple-Basic}Let $M$ be an $n-$dimensional
Riemannian manifold equipped with metric $dm^{2}$, continuous surface
density $\Phi_{S}$, and continuous volume density $\Phi_{V}$. There
exists a continuous conformal change of metric on $M$, $d\tilde{m}^{2}$,
and a positive continuous function, $\Psi$, such that the volume
and surface area of a region in $M$ with metric $dm^{2}$ and densities
$\Phi_{V}$ and $\Phi_{S}$ is the same as the volume and surface
area of the same region in $M$ with metric $d\tilde{m}^{2}$ and
simple density $\Psi$. \end{prop}
\begin{proof}
Take $d\tilde{m}^{2}=\left[\left(\frac{\Phi_{V}}{\Phi_{S}}\right)dm\right]^{2}$
and $\Psi=\frac{\Phi_{S}^{n}}{\Phi_{V}^{n-1}}$.
\end{proof}
One consequence of Proposition \ref{pro:Conformal-To-Simple-Basic}
is that when considering the isoperimetric problem in a manifold with
density one can always reduce to the case of simple density. One can
also always make a similar change of coordinates in order to work
with volume density or surface density. This reduction, however, does
not in general preserve the structure of a warped product, as clarified
by the following example:
\begin{example}
Let $X$ be a Riemannian manifold with metric $dx^{2}$ and $Y$ a
Riemannian manifold with metric $dy^{2}$, and let $g$ be a positive
continuous function on $X$. We examine Proposition \ref{pro:Conformal-To-Simple-Basic}
in the case where $M$ is the warped product $X\times_{g}Y$ with
metric $dm^{2}=dx^{2}+(g(x)dy)^{2}$ and continuous product densities
$\Phi_{S}^{X}\Phi_{S}^{Y}$ and $\Phi_{V}^{X}\Phi_{V}^{Y}$. The new
metric is given by 
\[
d\tilde{m}^{2}=\left(\frac{\Phi_{V}^{X}\Phi_{V}^{Y}}{\Phi_{S}^{X}\Phi_{S}^{Y}}\right)^{2}\left(dx^{2}+(g(x)dy)^{2}\right)
\]
and thus $M$ with metric $d\tilde{m}^{2}$ is no longer necessarily
a warped product of $X$ and $Y$ even after conformal changes of
metric in $X$ and $Y$. However, if instead we only move the $X$
component of the densities into the new metric on $M$ then we can
absorb this factor into the metric on $X$ and the warp factor $g(x)$
in order to obtain a new space given by a warped product $\tilde{X}\times_{\tilde{g}}Y$
(the tilde denoting the change of metric on $X$ and the change of
the warp factor $g$) with surface and volume densities differing
only in their $Y$ component and such that surface areas and volumes
of regions are the same as in the original warped product $M$. In
particular, if $\Phi_{S}^{Y}=\Phi_{V}^{Y}=1$ then this new space
is a warped product with simple density. 

If we are only interested in the vertical surface area, however, then
we can always obtain a simple comparison space in the form of a simple
product with surface density:\end{example}
\begin{prop}
\label{pro:ModelVertical}Let $L$ be a Riemannian manifold of dimension
$n-1$ with metric $dl^{2}$ and let $Z$ be the interval $(A,B),\; A<B\in\mathbb{R}\cup\pm\{\infty\}$
with the usual metric $dr^{2}$. Consider the warped product $Z\times_{g}L$
with continuous warp factor $g$ and continuous product densities
$\Phi_{S}^{Z}\Phi_{S}^{L}$ and $\Phi_{V}^{Z}\Phi_{V}^{L}$. For a
fixed $a\in(A,B)$, let $s(r)$ be the function 
\[
s(r)=\int_{a}^{r}\Phi_{V}^{Z}(t)g(t)^{n-1}dt,
\]
and let $\tilde{Z}$ be the interval $(s(A),s(B)).$ Let $\Psi(s)=\Phi_{S}^{Z}(r(s))g(r(s))^{n-1}$
and let $\tilde{L}$ be the space $L$ after conformal change of metric
$d\tilde{l}=(\left[\Phi_{V}^{L}\right]^{1/(n-1)}dl)$. Then, the map
$(r,l)\mapsto(s(r),l)$ from $Z\times_{g}L$ with densities $\Phi_{S}^{Z}\Phi_{S}^{L}$
and $\Phi_{V}^{Z}\Phi_{V}^{L}$ to the product $\tilde{Z}\times\tilde{L}$
with surface density $\Psi(s)\frac{\Phi_{S}^{L}(l)}{\Phi_{V}^{L}(l)}$
is a $C^{1}$ diffeomorphism that preserves volume and vertical surface
area. \end{prop}
\begin{rem*}
If $L$ with density $\Phi_{V}^{L}$ has finite volume then the coordinate
$s(r)$ of Proposition \ref{pro:ModelVertical} is a constant times
the signed volume of the region $(a,r)\times L$, and if fibers $\{r\}\times L$
have finite surface area then $\Psi(s)$ is a constant times the surface
area of the fiber $\{r(s)\}\times L$. \end{rem*}
\begin{proof}
We examine the local elements of volume in these two spaces: 
\[
dV_{Z\times_{g}L}=\Phi_{V}^{Z}(r)\Phi_{V}^{L}(l)g(r)^{n-1}drdL=dsd\tilde{L}=dV_{\tilde{Z}\times\tilde{L}}
\]
thus volume is preserved. For the local elements of vertical surface
area, we observe

\[
\Phi_{S}^{Z}(r)\Phi_{S}^{L}(l)g(r)^{n-1}dL=\Psi(s)\frac{\Phi_{S}^{L}(l)}{\Phi_{V}^{L}(l)}d\tilde{L},
\]
which completes the proof. 
\end{proof}

In Lemmas \ref{lem:Model-MainLemma} and \ref{lem:Model-MinVertAllRegions}
we give sufficient conditions for vertical fibers to minimize vertical
surface area in an important family of these model spaces. In order
to state our results in sufficient generality for later applications,
we need the notion of net volume. For two oriented hypersurfaces $H_{1}$
and $H_{2}$ in a Riemannian manifold $M$ we will say that $H_{1}-H_{2}$
bounds a signed oriented region $R$ if, in the language of geometric
measure theory (cf. \cite{morgangmt}), $R$ is an integral current
with boundary $H_{1}-H_{2}$. The integral current $R$ in this case
is just a formal sum $R=\sum n_{i}R_{i}$ of rectifiable sets $R_{i}$
of finite volume with $\sum n_{i}\partial R_{i}=H_{1}-H_{2}$. For
example, in $\mathbb{R}^{n}$ with a metric giving finite total volume,
two spheres $H_{1}$ and $H_{2}$ of different radii centered at the
origin with appropriate orientations give an annulus $A$ with boundary
$H_{1}-H_{2}$. The complement $\mathbb{R}^{n}\backslash A$ with
the natural orientation from $\mathbb{R}^{n}$ is equal (as a current)
to $\mathbb{R}^{n}-A$; its boundary is $H_{2}-H_{1}$, thus we conclude
that $H_{1}-H_{2}$ also bounds the current $-\mathbb{R}^{n}\backslash A=A-\mathbb{R}^{n}$
(in general, if $H_{1}$ and $H_{2}$ live in an oriented manifold
$M$ of finite volume then two regions bounded by $H_{1}-H_{2}$ differ
by a multiple of $M$). In any case, if $H_{1}-H_{2}$ bounds $R=\sum n_{i}R_{i}$,
the net volume of $R$ is $\sum n_{i}\mathrm{Vol}(R_{i})$, which
is well defined. Note that in the example, $A$ has positive volume
and $-\mathbb{R}^{n}\backslash A$ has negative volume, so that the
net volume bounded by $H_{1}-H_{2}$ depends on the choice of region
and can be positive, negative, or even zero. 
\begin{lem}
\label{lem:Model-MainLemma} Let $L$ be a Riemannian manifold of
dimension $n-1$ with metric $dl^{2}$ giving finite total $n-1$
dimensional measure and let $Z$ be the interval $(A,B),\; A,B\in\mathbb{R}\cup\pm\{\infty\}$
with the usual metric $dr^{2}$. In the product $Z\times L$ with
volume density 1 and convex surface density $\Psi(r)$, for any rectifiable
hypersurface $H$ and $r_{0}\in Z$ such that $H-\{r_{0}\}\times L$
is the boundary of a signed oriented region of net volume 0 that is
bounded away from $A$ and $B$ in almost every horizontal fiber $Z\times\{l\}$,
$|H|_{Vert}\geq|\{r_{0}\}\times L|$. 

If $\Psi$ does not approach $0$ at $B$ (resp. $A$) the condition
that the region be bounded away from $B$ (resp. $A)$ in almost every
fiber can be weakened to the region not containing an open interval
about $B$ (resp. $A$) in almost every fiber. \end{lem}
\begin{proof}
Suppose $H$ is such that $H-\{r_{0}\}\times L$ is the boundary of
a signed oriented region of net volume 0 that is bounded away from
$A$ and $B$ in almost every horizontal fiber $Z\times\{l\}$. By
translation of the interval $Z$, we can assume $r_{0}=0$. 

Let $R=R^{+}-R^{-}$ be the signed oriented region bounded by $H-\{0\}\times L$
of net volume 0 and bounded away from the origin and infinity in almost
every fiber $Z\times\{l\}$. Then, 
\[
0=\int_{R^{+}}dV-\int_{R^{-}}dV
\]
and by Fubini,
\[
\begin{array}{c}
0=\int_{L}h(l)dl\\
\mbox{where}\\
h(l)=h^{+}(l)+h^{-}(l),\\
h^{+}(l)=\int_{Z\times\{l\}\cap R^{+}}dr,\quad h^{-}(l)=-\int_{Z\times\{l\}\cap R^{-}}dr,
\end{array}
\]
and $h$, $h^{+}$, and $h^{-}$ are defined almost everywhere. Furthermore,
we claim that almost everywhere $h(l)$ is contained between the smallest
and largest $r$-coordinates of points in $Z\times\{l\}\cap H$ (which
is non-empty almost everywhere since for almost all $l$, $Z\times\{l\}\cap R$
is bounded away from $A$ and $B$): Since $Z\times\{l\}\cap R^{+}$
is bounded away from $B$ and all the $r$ coordinates in $R^{+}$
are positive, $h^{+}(l)$ is between $0$ and the largest $r$-coordinate
in $Z\times\{l\}\cap R^{+}$ which, if non-zero, is the largest $r$-coordinate
in $Z\times\{l\}\cap H$. Similarly, since $Z\times\{l\}\cap R^{-}$
is bounded away from $A$, $h^{-}(\Theta)$ is between the smallest
$r$-coordinate in $Z\times\{l\}\cap R^{-}$ and $0$, and the smallest
$r$-coordinate in $Z\times\{l\}\cap R^{-}$ is the smallest $r$-coordinate
in $Z\times\{l\}\cap H$ if it is non-zero. Thus, if both the intersections
$Z\times\{l\}\cap R^{+}$ and $Z\times\{l\}\cap R^{-}$ contain more
than the point with $r$-coordinate 0 then $h(l)$ is contained between
the smallest and largest $r$-coordinates of points in $Z\times\{l\}\cap H$.
If $Z\times\{l\}\cap R^{+}$ contains only the $r$-coordinate 0 then
it must be that $Z\times\{l\}\cap R^{-}$ contains the interval $(r_{max},0)$
where $r_{max}$ is the largest $r$-coordinate in $Z\times\{l\}\cap R^{-}$
which in this case is the largest $r$-coordinate in $Z\times\{l\}\cap H$,
and in particular $h^{-}(l)\leq r_{max}$ so that we again reach the
conclusion. We argue similarly if $Z\times\{l\}\cap R^{-}$ contains
only the $r$-coordinate 0. 

Since $H$ is rectifiable its intersections with the rays $Z\times\{l\}$
are $L$-almost everywhere transversal ($H$ differs from a countable
union of $C^{1}$ hypersurfaces by a set of $n-1$ dimensional Hausdorff
measure 0 and this holds for these hypersurfaces), and thus, by the
coarea formula,

\begin{eqnarray*}
|H|_{Vert} & = & \int_{L}\left(\sum_{r\in Z\times\{l\}\cap H}\Psi(r)\right)dL
\end{eqnarray*}
and since $\Psi$ is convex and $h$ is between the minimum and maximum
$r$-value in each fiber
\[
|H|_{Vert}\geq\int_{L}\Psi(h(l))dL,
\]
and by Jensen's inequality (applied to the normalization of the measure
$dL$ which has finite total measure),

\[
|H|_{Vert}\geq\int_{L}\Psi(0)dL,
\]
and the quantity on the right is the surface area of $\{0\}\times L$.

To prove the last statement of the Lemma we observe that in a fiber
where the region is neither bounded away from $B$ nor contains an
open interval about $B$, $H$ must intersect the fiber transversely
in infinitely many points in any neighborhood of $B$. Since $\Psi$
is convex and does not approach $0$ at $B$, it has a positive minimum
in a neighborhood of $B$ bounded away from $A$, and thus if there
is a set of positive measure of such fibers then $H$ has infinite
vertical surface area. Since we can assume that $H$ has finite vertical
surface area, we conclude the set of such fibers is of measure 0.
Since by hypothesis the set of fibers where the region contains an
open interval about $B$ has measure 0, almost everywhere the region
must be bounded away from $B$. \end{proof}
\begin{rem}
In the setting of Lemma \ref{lem:Model-MainLemma}, if $\Psi$ is
not convex then in general vertical fibers do not minimize vertical
surface area: Consider the cylinder $\mathbb{R}\times S^{1}$ with
non-convex smooth perimeter density $\Psi$. Since $\Psi$ is not
convex, there exist $r_{0},\, r_{1}\in\mathbb{R}$ such that $r_{0}<r_{1}$
and 
\[
\Psi\left(\frac{r_{0}+r_{1}}{2}\right)>\frac{\Psi(r_{0})+\Psi(r_{1})}{2}.
\]
Then the curve in $\mathbb{R}\times S^{1}$ given as the union of
the arcs $\{r_{0}\}\times[0,\pi]$ , $\{r_{1}\}\times[\pi,2\pi]$
and the radial segments $[r_{0},r_{1}]\times\{\pi\}$ and $[r_{0},r_{1}]\times\{0\}$
bounds net area 0 with the circle $\left\{ \frac{r_{0}+r_{1}}{2}\right\} \times S^{1}$
and has less vertical perimeter. This example can be generalized to
any interval and any type of vertical fiber -- just take two half
spaces of the vertical fiber at radii as chosen above and join them
along the set of horizontal fibers over their border.\end{rem}
\begin{lem}
\label{lem:Model-MinVertAllRegions}Let $L$ be a Riemannian manifold
of dimension $n-1$ with metric $dl^{2}$ giving finite total measure
on $L$ and let $Z$ be the interval $(A,B),\; A,B\in\mathbb{R}\cup\pm\{\infty\}$
with the usual metric $dr^{2}$. Consider the product $Z\times L$
with volume density 1 and continuous surface density $\Psi(r)$ with
$\Psi$ convex.
\begin{enumerate}
\item If $Z=(0,\infty)$ and $\lim_{r\rightarrow0}\Psi(r)=0$ then fibers
$\{r\}\times L$ minimize vertical surface area among hypersurfaces
bounding the same volume. 
\item If $Z=(-1,1)$, $\lim_{r\rightarrow-1}\Psi(r)=0$ and $\int_{0}^{1}\Psi(r)=\infty$,
then fibers $\{r\}\times L$ for $r\leq0$ minimize vertical surface
area among hypersurfaces of finite surface area bounding the same
volume.
\item If $Z=(-\infty,\infty)$ and $\lim_{r\rightarrow-\infty}\Psi(r)>0$
and $\lim_{r\rightarrow\infty}\Psi(r)>0$ then fibers $\{r\}\times L$
minimize vertical surface area among hypersurfaces with which they
bound net volume zero. 
\end{enumerate}
\end{lem}
\begin{proof}
(1):

Let $R$ be a region with rectifiable boundary with $|R|=|(0,r_{0})\times L|$
and suppose $|\partial R|_{Vert}<|\{r_{0}\}\times L|$. Since there
is infinite volume at infinity and $R$ has finite volume, almost
every fiber of $R$ does not contain an interval around infinity.
We have that $\lim_{r\rightarrow0}|\{r\}\times L|=0$ (from finite
surface area of fibers and the limit of $\Psi$), that $\lim_{r\rightarrow0}|(0,r)\times L|=0$,
and that the surface area $|\{r\}\times L|$ is a continuous function
of the volume $|(0,r)\times L|$, and thus we can take $r$ small
enough so that $|\partial(R\cup(0,r]\times L)|_{Vert}<|\{r'\}\times L|$
where $r'$ is such that $|R\cup(0,r]\times L|=|(0,r')\times L|$.
However, the convexity condition combined with the limit at $0$ implies
that $\Psi$ is non-decreasing and thus does not approach 0 at $\infty$
so we can apply Lemma \ref{lem:Model-MainLemma} to $\partial(R\cup(0,r]\times L)$
to obtain a contradiction.

(2):

We will show that for each region $R$ with rectifiable boundary and
finite surface area there exists a vertical fiber bounding the same
volume and having vertical surface area less than or equal to that
of $R$. The result will then follow because the convexity combined
with the limit and positivity of $\Psi$ imply that $\Psi$ is increasing
and thus the vertical fibers for $r\leq0$ have vertical surface area
strictly less than that of the fibers for $r>0$, and since there
is finite total volume with half the volume bounded by the fiber at
$r=0$, every volume can be bounded by a fiber with $r\leq0$. 

Let $R$ be a region with rectifiable boundary and finite surface
area. Suppose that for both $R$ and $R^{C}$ the sets of points such
that the fibers contain open intervals at infinity has positive measure.
Then for $r$ sufficiently large, the vertical fibers $\{r\}\times L\cap R$
have $n-1$ dimensional $L$-measure contained in a compact real interval
bounded away from 0 and the maximum measure of $L$. Since $L$ has
finite total measure, in this space isoperimetric regions exist for
all volumes, and since the isoperimetric profile is a continuous function,
there is an isoperimetric inequality on $L$ that implies for $r$
sufficiently large, say $r>1-\delta$, $\partial(\{r\}\times L\cap R)$
has $n-2$ dimensional surface area greater than a fixed $\epsilon>0$
viewed as a surface in $L$. In particular, by analyzing the horizontal
component of surface area, we conclude that $|\partial R|\geq\epsilon\int_{1-\delta}^{1}\Psi(r)dr>\infty$. 

Thus for either $R$ or $R^{C}$ almost every fiber does not contain
an open interval at infinity. If it is $R$, we can proceed as in
case \emph{(1)}. If it is $R^{C}$ then since the space has finite
total volume, $R^{C}$ also has finite volume and we can proceed as
in case \emph{(1) }with the region $R^{C}$. However, the same fiber
that bounds volume $|R^{C}|$ at $-1$ bounds volume $|R|$ at $1$,
and since $R$ and $R^{C}$ have the same boundary, we obtain the
conclusion for $R$. 

(3): 

Let $H$ be a rectifiable hypersurface such that $H-\{r\}\times L$
bounds the region $R=R^{+}-R^{-}$ with signed volume 0. In particular,
$R^{+}$ and $R^{-}$ both have finite volume, and since there is
infinite volume at both plus and minus infinity, as in (1) we obtain
that for almost every horizontal fiber $R$ does not contain an open
interval at $-\infty$ or $\infty$. Thus, we can apply Lemma \ref{lem:Model-MainLemma}
to obtain the result. \end{proof}
\begin{rem}
\label{rem:ModelMinimizers-Necessity}In Lemma \ref{lem:Model-MinVertAllRegions}
cases \emph{(1)} and \emph{(2)}, the condition on the limit of the
surface density is necessary -- otherwise, fibers bounding small volumes
will have a surface area bounded below by some constant so that for
sufficiently small volume they cannot be isoperimetric. In case \emph{(2)},
the condition on the integral of the surface density is necessary
to avoid regions bounded by horizontal surfaces $Z\times M$ for some
$n-2$ dimensional submanifold $M\subset L$ which have zero vertical
surface area - it guarantees that these regions have infinite surface
area (this is also why the extra hypothesis on finite surface area
is necessary in the statement). The conditions on the limits in case
\emph{(3) }is used in the proof to rule out certain regions but may
or may not be necessary.
\end{rem}
Using Proposition \ref{pro:ModelVertical}, we can lift Lemma \ref{lem:Model-MinVertAllRegions}
to a much more general setting:
\begin{thm}
\label{thm:GeneralFibersDensityMinimizeTanUnIso}Let $L$ be a Riemannian
manifold of dimension $n-1$ with metric $dl^{2}$ and let $Z$ be
the interval $(A,B),$ $A<0<B$ with standard metric $dr^{2}$. Consider
a warped product $Z\times_{g}L$ ($g$ continuous) with metric $dr{}^{2}+g(r)^{2}dl^{2}$
and continuous product surface density $\Psi_{S}(r)\Phi(l)$ and volume
density $\Psi_{V}(r)\Phi(l)$. Suppose that the surface area of fibers
$\{r\}\times L$ and the signed volume of any annulus $[0,r]\times L,\; A<r<B$
is finite and that the surface area of fibers $\{r\}\times L$ is
a convex function of the signed volume of the annulus $[0,r]\times L$.
\begin{enumerate}
\item If there is infinite total volume, $(A,0]\times L$ has finite volume,
and 
\[
\lim_{r\rightarrow A}|\{r\}\times L\}|=0
\]
 then fibers $\{r\}\times L$ minimize vertical surface area among
hypersurfaces bounding the same volume and thus are uniquely isoperimetric
for all volumes.
\item If there is finite total volume $V_{0}$, 
\[
\lim_{r\rightarrow A}|\{r\}\times L|=0,
\]
and 
\[
\int_{0}^{B}\Psi_{S}(r)g(r)^{n-2}dr=\infty
\]
then fibers $\{r\}\times L$ such that $|(A,r)\times L|\leq V_{0}/2$
minimize vertical surface area among hypersurfaces bounding the same
volume and having finite surface area and thus are uniquely isoperimetric
for all volumes.
\item If both $(A,0]\times L$ and $[0,B)\times L$ have infinite volume
and both 
\[
\lim_{r\rightarrow A}|\{r\}\times L|>0
\]
and 
\[
\lim_{r\rightarrow B}|\{r\}\times L|>0
\]
then fibers $\{r\}\times L$ minimize vertical surface area among
rectifiable hypersurfaces with which they bound net volume 0 and thus
uniquely minimize surface area among such surfaces. 
\end{enumerate}
\end{thm}
\begin{rem*}
By reflecting the interval we see statements 1 and 2 also hold reversing
the roles of $A$ and $B$. \end{rem*}
\begin{proof}
After applying Proposition \ref{pro:ModelVertical} and possibly shifting
or scaling the interval, we reduce to the corresponding cases of Lemma
\ref{lem:Model-MinVertAllRegions}: The component of the density depending
on $l$ is merged completely into the metric and, following the remark
after Proposition \ref{pro:ModelVertical}, the convexity of the surface
density on the resulting space is equivalent to the convexity of surface
area of fibers as a function of the volume of annuli. 

For \emph{(1)} and \emph{(3)}, the conditions on the limit translate
directly. 

For \emph{(2)}, the condition on the limit translates directly, however,
the condition on the integrals is not the same. This reflects the
fact that while vertical surface area is preserved in Proposition
\ref{pro:ModelVertical}, surface area is not, and so the regions
of finite surface area are not necessarily the same. However, the
condition given allows us to deduce by essentially the same argument
as in the proof of Lemma \ref{lem:Model-MinVertAllRegions}-\emph{(2)
}that regions of finite surface area contain open intervals about
$B$ only in a set of horizontal fibers of measure 0 where the measure
on $L$ is induced by the volume element on $L$, $\Phi(l)dL$, which
is the same as the measure on $\tilde{L}$ induced by $\tilde{dl}$.
After this, the rest of the proof goes through unchanged. 

That vertical fibers are uniquely isoperimetric follows because surface
area is greater than vertical surface area with equality only when
the surface is a union of vertical fibers. In cases \emph{(1)} and
\emph{(2)} the surface area of fibers is increasing so that multiple
fibers are always worse than a fiber of small radius bounding the
same volume, and in case \emph{(3)} multiple fibers are always worse
by convexity because for multiple fibers to bound net volume zero
with a given fiber when there is infinite volume at both the origin
and infinity some of them must lie on both sides of the given fiber. 
\end{proof}
It is possible to generalize Theorem \ref{thm:GeneralFibersDensityMinimizeTanUnIso}
to the case where only certain fibers minimize vertical surface area:
\begin{thm}
\label{thm:SingleFiberGeneralMinimizeTanIso}Let $L$ be a Riemannian
manifold of dimension $n-1$ with metric $dl^{2}$ and let $Z$ be
the interval $(A,B),$ $A<0<B$ with standard metric $dr^{2}$. Consider
a warped product $Z\times_{g}L$ ($g$ continuous) with metric $dr{}^{2}+g(r)^{2}dl^{2}$
and continuous product surface density $\Psi_{S}(r)\Phi(l)$ and volume
density $\Psi_{V}(r)\Phi(l)$. Suppose that the surface area of fibers
$\{r\}\times L$ and the signed volume of any annulus $[0,r]\times L,\; A<r<B$
is finite.
\begin{enumerate}
\item Suppose there is infinite total volume and $(A,0]\times L$ has finite
volume. Let $F$ be the function that sends a volume $V$ to the surface
area of the unique vertical fiber $\{r\}\times L$ such that $|(A,r]\times L|=V$.
If there is a positive function $\tilde{F}$ and a $V_{0}$ such that
$\tilde{F}\leq F$, $\tilde{F}$ is convex, $\lim_{V\rightarrow0}\tilde{F}(V)=0$
and $\tilde{F}(V_{0})=F(V_{0})$, then the fiber $\{r\}\times L$
such that $|(A,r)\times L|=V_{0}$ minimizes vertical surface area
among surfaces bounding the same volume and thus is uniquely isoperimetric
for volume $V_{0}$. 
\item Suppose both $(A,0]\times L$ and $[0,B)\times L$ have infinite volume.
Let $F$ be the function that sends a signed volume $V$ to the surface
area of the unique vertical fiber $\{r\}\times L$ such that $|[0,r]\times L|=V$.
If there is a positive function $\tilde{F}$ and a $V_{0}$ such that
$\tilde{F}\leq F$, $\tilde{F}$ is convex, $\lim_{V\rightarrow-|(A,0]\times L|}\tilde{F}(V)>0$,
$\lim_{V\rightarrow|[0,B)\times L|}\tilde{F}(V)>0$, and $\tilde{F}(V_{0})=F(V_{0})$
then the fiber $\{r\}\times L$ such that $|(0,r)\times L|=V_{0}$
minimizes vertical surface area among rectifiable hypersurfaces with
which it bounds net volume 0 and thus uniquely minimizes surface area
among such surfaces. 
\end{enumerate}
\end{thm}
\begin{rem*}
It is possible to give a version of Theorem \ref{thm:SingleFiberGeneralMinimizeTanIso}
in the finite volume case (case \emph{(2) }of Theorem \ref{thm:GeneralFibersDensityMinimizeTanUnIso}),
however the condition on the surface area density makes the hypothesis
more complicated to state and we have found no interesting examples
where it applies, and thus we omit it here. \end{rem*}
\begin{proof}
We can replace the $r$-component of the surface density with a new
$r$-component so that the surface area of vertical fibers is now
given by $\tilde{F}$. By our hypothesis on $\tilde{F}$ we can apply
Theorem \ref{thm:GeneralFibersDensityMinimizeTanUnIso} to this new
space. Since $\tilde{F}\leq F$, the new surface density is everywhere
less than or equal to the original surface density and so the surface
area (resp. vertical surface area) of any region considered as a region
in this new space is less than or equal to its surface area (resp.
vertical surface area) as a region in the original space. Since $\tilde{F}(V_{0})=F(V_{0})$,
the fiber in question has the same surface area and vertical surface
area in both spaces. Volume is the same for all regions in both spaces,
and thus this fiber also satisfies the conclusions of Theorem \ref{thm:GeneralFibersDensityMinimizeTanUnIso}
in the original space.
\end{proof}
We deduce the following useful corollary:
\begin{cor}
\label{cor:LargeVerticalFibers}Let $L$ be a Riemannian manifold
of dimension $n-1$ with metric $dl^{2}$ and let $Z$ be the interval
$(0,B)$ with standard metric $dr^{2}$. Consider a warped product
$Z\times_{g}L$ ($g$ continuous) with metric $dr{}^{2}+g(r)^{2}dl^{2}$
and continuous product surface density $\Psi_{S}(r)\Phi(l)$ and volume
density $\Psi_{V}(r)\Phi(l)$. Suppose that the surface area of fibers
$\{r\}\times L$ and the volume of annuli $(0,r)\times L,\;0<r<B$
are finite, and that there is infinite total volume. Let $F$ be the
function that sends a volume $V$ to the surface area of the unique
vertical fiber $\{r\}\times L$ such that $|(0,r]\times L|=V$. If
$F$ is eventually convex, $F$ is bounded below by a line through
the origin of positive slope, and $\lim_{V\rightarrow\infty}F'(V)=\infty$
then for sufficiently large $r$, vertical fibers $\{r\}\times L$
are isoperimetric. \end{cor}
\begin{proof}
In order to apply Theorem \ref{thm:SingleFiberGeneralMinimizeTanIso},
it suffices to find a convex function $\tilde{F}$ such that $\tilde{F}\leq F$,
$\tilde{F}(V)=F(V)$ for $V$ sufficiently large, and $\lim_{V\rightarrow0}\tilde{F}(V)=0$.
Since $F$ is eventually convex and $\lim_{V\rightarrow\infty}F'(V)=\infty$,
there is a $V_{1}$ such that $F$ is convex for $V\geq V_{1}$ and
the tangent line to the graph of $F$ at $V_{1}$ bounds $F$ from
below and has a non-negative $V$-intercept. Denote by $l_{1}$ this
tangent line and denote by $l_{0}$ a line through the origin bounding
$F$ from below and let $V_{0}$ be the $V$-coordinate of the intersection
of $l_{0}$ and $l_{1}$ (forcibly $V_{0}\leq V_{1}$). Then we define
\[
\tilde{F}(V)=\begin{cases}
l_{0}(V) & \;0<V\leq V_{0}\\
l_{1}(V) & V_{0}\leq V\leq V_{1}\\
F(V) & V_{1}\leq V
\end{cases},
\]
which has the desired properties. 
\end{proof}

\section{$\mathbb{R}^{n}-\{0\}$ with radial density\label{sec:RnWithRadialDensity}}

We apply the results of Section \ref{sec:Vertical-surface-area-Warped-Products}
in the motivating case of $\mathbb{R}^{n}-\{0\}$ considered as the
warped product $(0,\infty)\times_{r}S^{n-1}$ with radial density.
In this context, we refer to vertical surface area as tangential surface
area, following \cite{diazetal}. We note that the results of this
section extend to $(0,\infty)\times_{r}S^{n-1}$ with continuous product
surface density $\Psi_{S}(r)\Phi(\Theta)$ and volume density $\Psi_{V}(r)\Phi(\Theta)$
(note that $\Phi$ must be the same for both the surface and volume
density), as well as to warped products $(0,\infty)\times_{r}K$ for
any compact Riemannian manifold $K$ with densities of the same form. 
\begin{thm}
\label{thm: Isoperimetric Regions - General}Consider $\mathbb{R}^{n}-\{0\}$
with continuous radial surface density $\Psi_{S}$ and continuous
radial volume density $\Psi_{V}$ such that the surface area of a
sphere of radius $r$ is a convex function of the signed volume of
the annulus $[1,r]\times S^{n-1}$.
\begin{enumerate}
\item If there is infinite total volume and either 

\begin{enumerate}
\item there is finite volume at the origin and $\lim_{r\rightarrow0}|\partial B_{r}|=0$
or 
\item there is finite volume at infinity and $\lim_{r\rightarrow\infty}|\partial B_{r}|=0$, 
\end{enumerate}

then spheres about the origin minimize tangential surface area among
hypersurfaces bounding the same volume and thus are uniquely isoperimetric
for all volumes.

\item If there is finite total volume and 
\[
\lim_{r\rightarrow0}|\partial B_{r}|=0\;\mbox{(resp. }\lim_{r\rightarrow\infty}|\partial B_{r}|=0\mbox{),}
\]
then spheres about the origin bounding volume less than or equal to
half the total volume of the space at the origin (resp. at infinity)
minimize tangential surface area among hypersurfaces of finite surface
area bounding the same volume and thus are uniquely isoperimetric
for all volumes.
\item If there infinite volume at both the origin and infinity and both
\[
\lim_{r\rightarrow0}|\partial B_{r}|>0\;\mbox{and }\lim_{r\rightarrow\infty}|\partial B_{r}|>0,
\]
then any sphere about the origin $S$ minimizes vertical surface area
and thus uniquely minimizes surface area among rectifiable hypersurfaces
$H$ bounding net volume zero with $S$. 
\end{enumerate}
\end{thm}
\begin{proof}
This is just Theorem \ref{thm:GeneralFibersDensityMinimizeTanUnIso}
applied in this setting. To obtain, for example, \emph{(1)-(b)}, we
consider the warped product with the interval reflected about 0. The
only case that needs extra verification is \emph{(2), }for which we
must verify that $\int_{1}^{\infty}\Phi_{S}(r)r^{n-2}dr=\infty$ (the
other case following similarly). However, by the convexity condition
and the fact that $\lim_{r\rightarrow0}|\partial B_{r}|=0$, $|\partial B_{r}|=c\cdot r^{n-1}\Psi_{S}(R)$
must be increasing, and thus 
\[
\int_{1}^{\infty}r^{n-2}\Psi_{S}(r)\geq\int_{1}^{\infty}\frac{1}{r}\Psi_{S}(1)=\infty.
\]

\end{proof}
We obtain the following simplified statement in the case of simple
density:
\begin{thm}
\label{thm:IsoperimetricRegionsSimpleDensity}Consider $\mathbb{R}^{n}-\{0\}$
with continous simple radial density $e^{\phi}$, where $|\partial B{}_{r}|$
is a log-convex function of $r$. 
\begin{enumerate}
\item If there is infinite total volume and finite volume at infinity then
spheres about the origin minimize tangential surface area among hypersurfaces
bounding the same volume and thus complements of balls about the origin
are uniquely isoperimetric for all volumes. 
\item If there is finite total volume then spheres about the origin bounding
volume greater than half the volume of the space at the origin minimize
tangential surface area among hypersurfaces bounding the same volume
and thus are uniquely isoperimetric for all volumes. 
\item If there infinite volume at both the origin and infinity and $\lim_{r\rightarrow\infty}|\partial B_{r}|>0$
then any sphere about the origin $S$ minimizes vertical surface area
and thus uniquely minimizes surface area among rectifiable hypersurfaces
$H$ bounding net volume zero with $S$. 
\end{enumerate}
\end{thm}
\begin{rem*}
For a smooth simple density the log-convexity of $|\partial B{}_{r}|$
is equivalent to $\phi''\geq\dfrac{n-1}{r^{2}}$. Thus, for a smooth
simple density the cases of Theorem \ref{thm: Isoperimetric Regions - General}
where $\lim_{r\rightarrow0}|\partial B_{r}|=0$ cannot occur because
this gives a contradiction.\end{rem*}
\begin{proof}
We apply Theorem \ref{thm: Isoperimetric Regions - General}. The
convexity conditions are easily checked to be equivalent. Thus, all
that is left to verify are the limit conditions. In the first two
cases we must show that spheres of large radius have surface area
approaching zero. However, finite volume at infinity gives us that
$|\partial B_{r}|$ is decreasing with $\lim_{r\rightarrow\infty}|\partial B_{r}|=0$
-- indeed, by the log-convexity condition, if $|\partial B_{r}|$
is ever non-decreasing then it is always non-decreasing and thus there
is infinite volume at infinity, and if it is decreasing but not decreasing
to 0 then there is also infinite volume at infinity. In the last case
$\lim_{r\rightarrow0}|\partial B_{r}|>0$ because there is infinite
volume at the origin. 
\end{proof}
Diaz \emph{et al.} \cite[Thm 7.4]{diazetal} used the technique of
tangential surface area to recover a result of Betta et al. on when
spheres about the origin are isoperimetric in the case of surface
density. We restate it here, noting it follows again as a corollary
of the more general Theorem \ref{thm: Isoperimetric Regions - General}.
The modified convexity condition is obtained by subtracting off the
value at the origin and noting that spheres are isoperimetric in $\mathbb{R}^{n}$
with constant density.
\begin{thm}[{Surface density, \cite[Thm 4.3]{bettaEtAl}}]
\label{thm:Isoperim-Surface Area density} In $\mathbb{R}^{n}$ with
non-decreasing radial surface density $\Psi$ such that 
\[
(\Psi(r^{1/n})-\Psi(0))r^{1-1/n}
\]
is convex, spheres about the origin are isoperimetric. 
\end{thm}

\begin{rem}[Volume density]
 In $\mathbb{R}^{n}-\{0\}$ with volume density $e^{\phi},$ the
convexity condition of Theorem \ref{thm: Isoperimetric Regions - General}
becomes $\phi'(r)\leq-1/r$ . However, any decreasing volume density
can be handled with simpler arguments (see \cite{morganBlogLogConvex}). 

\end{rem}
\begin{example}
\label{exa:IsoperimExamples}We present here some applications of
Theorems \ref{thm: Isoperimetric Regions - General} and \ref{thm:IsoperimetricRegionsSimpleDensity}.
\begin{enumerate}
\item Diaz \emph{et al.} \cite[ Prop 7.5]{diazetal} used the technique
of tangential surface area to prove that in $\mathbb{R}^{n}-\{0\}$
with simple density $r^{p}$, $p<-n$, spheres about the origin are
isoperimetric bounding volume at infinity. Using Theorem \ref{thm:IsoperimetricRegionsSimpleDensity}
we can extend this family to $r^{p}e^{\phi(r)}$ with $\phi''\geq0$
and either $\phi'(r)\leq0$ and $p<-n$ or $\phi'(r)\leq-\epsilon<0$
and $p=-n$. Notably, at the origin these densities are neither log-convex
nor smooth. Figure \ref{fig:GraphsFPsi} \emph{(d) }shows the graph
of surface area of spheres as a function of volume for one density
in this family.
\item Theorem \ref{thm:IsoperimetricRegionsSimpleDensity} also allows us
to generalize the first example further to $\mathbb{R}^{n}-\{0\}$
with simple density $r^{p}e^{\phi(r)}$, $-n<p\leq-(n-1)$, $\phi''\geq0$
and $\phi'\leq-\epsilon<0$ (for example $\phi(r)=-ar+b,$ $a>0$),
which gives finite total volume: in these spaces, spheres bounding
half or more of the volume of the space at the origin are isoperimetric. 
\item Theorem \ref{thm:IsoperimetricRegionsSimpleDensity} gives a similar
result in $\mathbb{R}^{n}-\{0\}$ with simple density $r^{-n}$, which
has infinite volume at both the origin and infinity: in this space,
any sphere about the origin $S$ minimizes surface area among all
rectifiable hypersurfaces bounding net volume zero with $S$. 
\item Consider $\mathbb{R}^{n}-\{0\}$ with volume density $\Psi_{V}$ such
that there is finite volume at the origin and infinite volume at infinity
and surface area density $\Psi_{S}(r)=G(|B_{r}|)/r^{n-1}$ where $G$
is convex and approaches zero at the origin. Then the surface area
of balls about the origin as a function of their volume is proportional
to $G$ and thus convex so we can apply Theorem \ref{thm: Isoperimetric Regions - General}
to show balls about the origin are isoperimetric. We can use this
method to construct many specific examples, for instance, $\Psi_{V}=r^{m}$
and $\Psi_{S}=r^{k}$ with $m\geq0$ and $k\geq m+1$. This example
provides a generalization to higher dimensions of part of the result
of Diaz et. al. \cite[Thm. 4.17 and Prop. 4.23]{diazetal} on isoperimetric
regions in sectors with simple density $r^{p}$, which after a conformal
change of coordinates is equivalent to $\mathbb{R}^{2}$ with differing
perimeter and area densities. 
\item Instead of density, we can can consider a conformal change of metric
(which is equivalent to certain differing densities on surface area
and volume): Consider $\mathbb{R}^{2}-\{0\}$ with metric $e^{\phi}ds$
where $ds$ is the standard Euclidean metric and $\phi$ is a function
of $r$. A calculation shows that the perimeter of circles as a function
of their area is convex at the area of the circle of radius $r$ if
and only if 
\[
\phi''(r)\geq\phi'(r)^{2}+\frac{\phi'(r)}{r}+\frac{1}{r^{2}}.
\]
In particular, Theorem \ref{thm: Isoperimetric Regions - General}-(3)
applies with $\phi=-\log r$ which gives the punctured plane with
metric $\frac{1}{r}\cdot ds$, i.e. a cylinder (we can also obtain
this result directly by applying Theorem \ref{thm: Isoperimetric Regions - General}-(3)
to $(-\infty,\infty)\times S^{1}$ with density 1). By adding $\phi'(r)/r$
to each side we can rewrite the convexity condition as
\[
\kappa(r)\geq\left(\phi'(r)+\frac{1}{r}\right)^{2}
\]
where $\kappa$ is the Gaussian curvature at radius $r$. Howards
\emph{et al. \cite[Sec. 9]{HHM-IsoperimetricProblemSurfaces}} give
other results on the isoperimetric problem in $\mathbb{R}^{2}$ with
non-Euclidean metrics, in particular showing that circles are isoperimetric
if the metric is smooth and has curvature that decreases with radius. 
\end{enumerate}
\end{example}

\begin{example}
In $\mathbb{R}^{n}-\{0\}$ with simple density $r^{p}\;-n<p<0$, isoperimetric
regions do not exist \cite[Prop. 7.3]{diazetal}. Although the convexity
condition holds for $-n<p\leq-(n-1)$, Theorem \ref{thm: Isoperimetric Regions - General}
does not apply: there is infinite total volume with finite volume
at the origin but the surface area of small spheres does not go to
0. Indeed, this is always the case for a simple density satisfying
these volume hypotheses. In particular, from non-existence for $p=-(n-1)$,
we see that the condition in Theorem \ref{thm: Isoperimetric Regions - General}-(1)
that the surface area goes to zero cannot be weakened to the condition
that the surface area of small spheres remains bounded (as noted already
in remark \ref{rem:ModelMinimizers-Necessity} for our model spaces). 
\end{example}

\subsection{Large balls in $\mathbb{R}^{n}-\{0\}$}

Using Corollary \ref{cor:LargeVerticalFibers} we generalize Kolesnikov
and Zhdanov's \cite[Prop. 4.7]{kolesnikovRadiallySymmetric} results
on large balls about the origin in $\mathbb{R}^{n}$ with simple density.
In particular, we obtain that in $\mathbb{R}^{n}$ with simple density
non-singular at the origin and exhibiting a type of eventually strict
log-convexity, spheres about the origin bounding large volume are
uniquely isoperimetric (Corollary \ref{cor:LargeBalls-SimpleDensity}).
In Example \ref{exa:LargeBalls} we demonstrate several densities
where Theorem \ref{thm:SingleFiberGeneralMinimizeTanIso}, Corollary
\ref{cor:LargeVerticalFibers} or Corollary \ref{cor:LargeBalls-SimpleDensity}
can be applied. We note that some of these more general results may
be attainable using the divergence theorem methods of Kolesnikov and
Zhdanov, however, as our proof seems to offer a clearer geometric
picture, we have not explored this possibility. 
\begin{cor}
\label{cor:LargeBalls-SimpleDensity}In $\mathbb{R}^{n}$ with continuous
radial density $\Psi=e^{\phi(r)}$, if there exist $r_{0}>0$ and
$\epsilon>0$ such that $\phi$ is twice continuously differentiable
with $\phi''(r)\geq\frac{\epsilon}{r}$ for all $r>r_{0}$, then large
spheres about the origin are uniquely isoperimetric.\end{cor}
\begin{rem*}
In fact, it suffices to have $\phi'$ go to infinity and $\phi''$
eventually greater than $\frac{n-1}{r^{2}}$.\end{rem*}
\begin{proof}
Let $F$ be the function mapping the volume of a ball about the origin
in $\mathbb{R}^{n}$ with simple density $\Psi$ to its surface area.
Since $\Psi$ is continuous and 
\[
0<\Psi(0)<\infty,
\]
there is a positive constant $C$ such that near the origin the surface
area of a ball is greater than $C/r$ times its volume. Thus, $F$
has a vertical asymptote at volume 0 (c.f. Figure \ref{fig:GraphsFPsi}-\emph{(a)}).
Now, for $V$ sufficiently large,
\[
F'(V)=\left(\frac{(n-1)}{r}+\phi'(r)\right)
\]

\[
F''(V)=\left(\phi''(r)-\frac{n-1}{r^{2}}\right)\cdot\frac{|S_{1}|}{|S_{r}|}
\]
where $r$ is the radius of the ball of volume $V$, and thus $F$
is eventually convex. Furthermore since $\phi'$ becomes arbitrarily
large, $\lim_{V\rightarrow\infty}F'(V)=\infty$. Since $F$ is positive
outside of zero, continuous, and eventually increasing, it has an
infimum greater than zero on any interval bounded away from 0, and
since $F$ also has a vertical asymptote at volume 0, there must exist
a line of positive slope through the origin bounding $F$ from below.
Finally, since $\Psi(0)$ being finite implies that there is finite
volume at the origin we can apply Corollary \ref{cor:LargeVerticalFibers}. \end{proof}
\begin{example}
\label{exa:LargeBalls}We give some applications of Theorem \ref{thm:SingleFiberGeneralMinimizeTanIso},
Corollary \ref{cor:LargeVerticalFibers}, and Corollary \ref{cor:LargeBalls-SimpleDensity}:
\begin{enumerate}
\item In $\mathbb{R}^{n}$ with simple density $e^{r^{\alpha}}$, $\alpha>1$,
Corollary \ref{cor:LargeBalls-SimpleDensity} shows that large spheres
about the origin are isoperimetric (originally shown by Kolesnikov
and Zhdanov \cite[Prop. 4.7]{kolesnikovRadiallySymmetric}). 
\item In $\mathbb{R}^{n}$ with simple density $e^{p(r)}$ with $p$ a polynomial
of degree greater than or equal to 2 and positive leading coefficient,
Corollary \ref{cor:LargeBalls-SimpleDensity} shows that large spheres
about the origin are isoperimetric. This include, for example, $\mathbb{R}^{n}$
with simple density $e^{r^{2}-2r+2}$, where for small volumes isoperimetric
regions are approximate balls centered on the unit circle \cite{morganBlogLogConvex,Morgan-PubLogConvexDensity}.
Thus we obtain examples of spaces where spheres about the origin are
stable but only isoperimetric for certain volumes. 
\item In $\mathbb{R}^{3}$ with density $\Psi_{V}=e^{r}$, $\Psi_{S}=e^{r^{8}}$,
large spheres about the origin are isoperimetric: It is easy to see
from the graph of surface area as a function of volume (Figure \ref{fig:GraphsFPsi}
\emph{(c)}) that an appropriate convex function exists and thus we
can apply Theorem \ref{thm:SingleFiberGeneralMinimizeTanIso}. Alternatively,
one can show directly that Corollary \ref{cor:LargeVerticalFibers}
applies.
\item As pointed out by Morgan \cite{Morgan-exprAlpha}, in $\mathbb{R}^{n}$
with simple density $e^{r^{\alpha}}$, $\alpha<0$, small spheres
about the origin minimize tangential surface area among hypersurfaces
with which they bound volume 0 and uniquely minimize surface area
among such hypersurfaces. Indeed, to apply Theorem \ref{thm:SingleFiberGeneralMinimizeTanIso}-\emph{(2)}
we can produce an appropriate convex function using the same ideas
as in the proof of Corollary \ref{cor:LargeVerticalFibers}. We can
also see this graphically in Figure \ref{fig:GraphsFPsi} \emph{(b)}. 
\end{enumerate}
\end{example}
\begin{figure}[t]
\begin{tabular}{c}
\begin{tabular}{>{\raggedright}p{0.4\textwidth}>{\raggedright}p{0.4\textwidth}}
\begin{tabular}{>{\centering}p{0.4\textwidth}}
\includegraphics[width=0.38\textwidth]{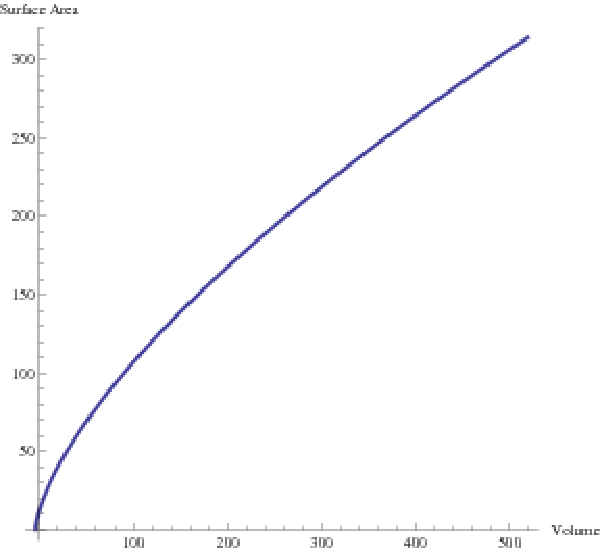}

$\Psi_{V}=\Psi_{S}=1$. \tabularnewline
\end{tabular} & %
\begin{tabular}{>{\centering}p{0.4\textwidth}}
\includegraphics[width=0.38\textwidth]{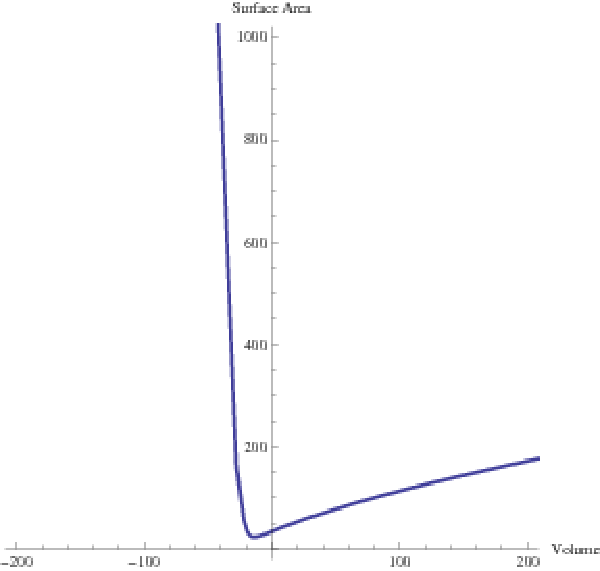}

$\Psi_{V}=\Psi_{S}=e^{r^{-1}}$. \tabularnewline
\end{tabular}\tabularnewline
\emph{(a) }The vertical asymptote at the origin is typical of densities
bounded away from 0 and infinity at the origin, as in the proof of
Corollary \ref{cor:LargeBalls-SimpleDensity}. & \emph{(b) }This space has infinite volume at both the origin and infinity,
and one can visually verify from the graph that Theorem \ref{thm:SingleFiberGeneralMinimizeTanIso}-\emph{(2)}
applies for small spheres which thus minimize surface area among surfaces
with which they bound net volume 0 (see Example \ref{exa:LargeBalls}). \tabularnewline[0.3cm]
\end{tabular}\tabularnewline
\noalign{\vskip\doublerulesep}
\begin{tabular}{>{\raggedright}p{0.4\textwidth}>{\raggedright}p{0.4\textwidth}}
\begin{tabular}{>{\centering}p{0.4\textwidth}}
\includegraphics[width=0.38\textwidth]{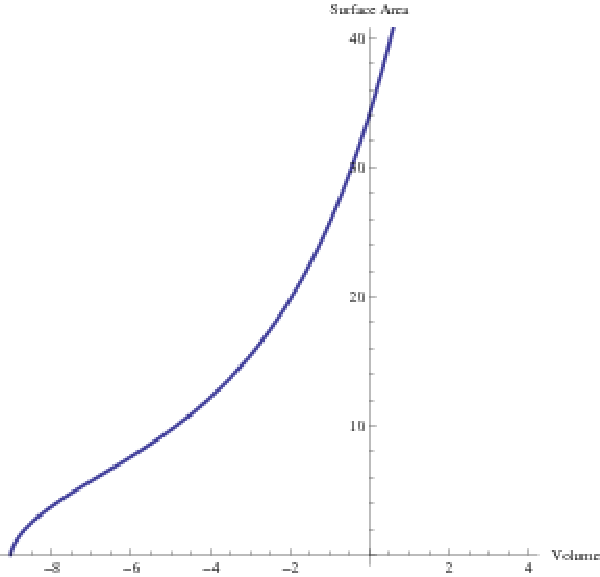}

$\Psi_{V}=e^{r},\;\Psi_{S}=e^{r^{8}}$. \tabularnewline
\end{tabular} & %
\begin{tabular}{>{\centering}p{0.4\textwidth}}
\includegraphics[width=0.38\textwidth]{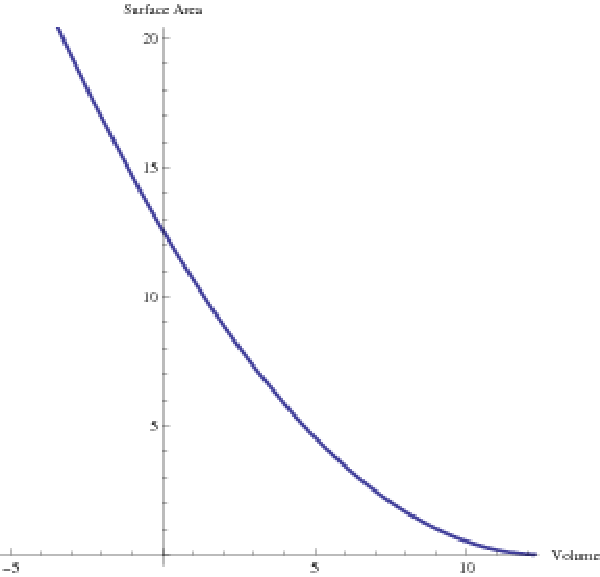}

$\Psi_{V}=\Psi_{S}=r^{-4}$. \tabularnewline
\end{tabular}\tabularnewline
\emph{(c)} One can visually verify from the graph that for this density
the hypotheses of Theorem \ref{thm:SingleFiberGeneralMinimizeTanIso}
are satisfied for large spheres which are thus isoperimetric (see
Example \ref{exa:LargeBalls}).\\ & \emph{(d) }At the origin this density is singular and log-convexity
fails severely (the log of the density tends to positive infinity
at the origin), but ball complements are isoperimetric for all volumes
(see Example \ref{exa:IsoperimExamples}).\\\tabularnewline[\doublerulesep]
\end{tabular}\tabularnewline
\end{tabular}

\caption{\label{fig:GraphsFPsi}Graph of the surface area of the sphere about
the origin of radius $r$ as a function of the signed volume of the
annulus $[1,r]\times S^{2}$ for various densities on $\mathbb{R}^{3}$.}
\end{figure}

\bibliographystyle{plain}
\bibliography{REFMWD}

\end{document}